\documentclass[a4paper,12pt]{amsart}

\usepackage{amsfonts}
\usepackage{amsmath}
\usepackage{amssymb}

\usepackage{mathrsfs}
\usepackage{hyperref}

\numberwithin{equation}{section}
\theoremstyle{plain}
\newtheorem{thm}{Theorem}[section]
\newtheorem{prop}[thm]{Proposition}

\newtheorem{lemma}[thm]{Lemma}

\theoremstyle{definition}
\newtheorem{defi}[thm]{Definition}

\newcommand{\be}{\begin{equation}}
\newcommand{\ee}{\end{equation}}

\def\R{{\mathbb R}}
\def\D{{\mathbb D}}

\def\N{{\mathbb N}}

\def\C{{\mathbb C}}
\def\S3{{{\mathbb S}^3}}
\def\SU2{{{\rm SU}(2)}}
\def\Rn{{{\mathbb R}^n}}

\def\Gh{{\widehat{G}}}

\def\HS{{\mathtt{HS}}}

\def\p#1{{\left({#1}\right)}}

\def\br#1{{\left[{#1}\right]}}
\def\jp#1{{\left\langle{#1}\right\rangle}}

\def\wt#1{{\widetilde{#1}}}

\def\Hcal{{\mathcal H}}

\def\Hcal{{\mathcal H}}

\def\va{\varphi}

\DeclareMathOperator{\Tr}{Tr}

\def\C{{\mathbb C}}

\def\Rn{{\mathbb R}^n}
\def\R2n{{\mathbb R}^{2n}}

\def\S{{\mathcal S}}
\def\D{{\mathcal D}}

\def\B{{\mathcal B}}

\def\Rn{{\mathbb R}^n}

\def\C{{\mathbb C}}

\def\R2{{\mathbb R}^2}
\def\R2n{{\mathbb R}^{2n}}

\def\S{{\mathcal S}}
\def\D{{\mathcal D}}

\def\B{{\mathcal B}}
\def\L{{\mathcal L}}

\def\Ghs{{\Gh_{*}}}
\def\sumxi{\sum_{[\xi]\in\Gh}}

\def\sumij{\sum_{i,j=1}^{d_\xi}}
\def\paal{{\partial^{\alpha}}}

\begin{document}

\title[Gevrey functions and ultradistributions on compact Lie groups]
{Gevrey functions and ultradistributions on compact Lie groups and
homogeneous spaces}

\author[Aparajita Dasgupta]{Aparajita Dasgupta}
\address{
  Aparajita Dasgupta:
  \endgraf
  Department of Mathematics
  \endgraf
  Imperial College London
  \endgraf
  180 Queen's Gate, London SW7 2AZ 
  \endgraf
  United Kingdom
  \endgraf
  {\it E-mail address} {\rm a.dasgupta@imperial.ac.uk}
  }

\author[Michael Ruzhansky]{Michael Ruzhansky}
\address{
  Michael Ruzhansky:
  \endgraf
  Department of Mathematics
  \endgraf
  Imperial College London
  \endgraf
  180 Queen's Gate, London SW7 2AZ 
  \endgraf
  United Kingdom
  \endgraf
  {\it E-mail address} {\rm m.ruzhansky@imperial.ac.uk}
  }

\thanks{The first author was supported by the Grace-Chisholm Young Fellowship from London Mathematical Society. The second
 author was supported by the EPSRC
 Leadership Fellowship EP/G007233/1.}
\date{\today}


\begin{abstract}
In this paper we give global characterisations of Gevrey-Roumieu and Gevrey-Beurling
spaces of ultradifferentiable functions on compact Lie groups in terms of the
representation theory of the group and the spectrum of the Laplace-Beltrami
operator. Furthermore, we characterise their duals, the spaces of corresponding
ultradistributions. For the latter, the proof is based on first obtaining
the characterisation of their $\alpha$-duals in the sense of K\"othe and
the theory of sequence spaces. We also give the corresponding characterisations on
compact homogeneous spaces.
\end{abstract}

\maketitle

\section{Introduction}

The spaces of Gevrey ultradifferentiable functions are well-known on $\Rn$ and their
characterisations exists on both the space-side and the Fourier transform side,
leading to numerous applications in different areas. The aim
of this paper is to obtain global characterisations of the spaces of Gevrey ultradifferentiable
functions and of the spaces of ultradistributions using the eigenvalues
of the Laplace-Beltrami operator $\L_{G}$ (Casimir element) on the compact Lie group $G$.
We treat both the cases of Gevrey-Roumieu and Gevrey-Beurling functions, and 
the corresponding spaces of ultradistributions, which are their topological duals 
with respect to their inductive and projective limit topologies, respectively.

If $M$ is a compact homogeneous space, let $G$ be its motion group and $H$ a 
stationary subgroup at some point, so that $M\simeq G/H.$
Our results on the motion group $G$ will yield the corresponding characterisations
for Gevrey functions and ultradistributions on the homogeneous space $M$.
Typical examples are the real spheres ${\mathbb S}^{n}={\rm SO}(n+1)/{\rm SO}(n)$,
complex spheres
(complex projective spaces) $\mathbb C\mathbb P^{n}={\rm SU}(n+1)/{\rm SU}(n)$,
or quaternionic projective spaces 
$\mathbb H\mathbb P^{n}$. 

Working in local coordinates and treating $G$ as a manifold
the Gevrey(-Roumieu) class $\gamma_{s}(G)$, $s\geq 1$, is the space of 
functions $\phi\in{C^{\infty}(G)}$ such that in every local coordinate
chart its local representative, say $\psi\in C^{\infty}(\Rn)$, is such that 
there exist constants $A>0$ and $C>0$ 
such that for all multi-indices $\alpha,$ we have that
$$
{|\partial^{\alpha}\psi(x)|\leq C A^{|\alpha|}\left(\alpha !\right)^{s}}
$$
holds for all $x\in\Rn$. By the chain rule one readily sees that this class is
invariantly defined on (the analytic manifold) $G$ for $s\geq 1.$
For $s=1$ we obtain the class of analytic functions.
This behaviour can be characterised on the 
Fourier side by being equivalent to the condition that there exist $B>0$ 
and $K>0$ such that
$$|\widehat{\psi}(\eta)|\leq K e^{-B\jp{\eta}^{1/s}}$$ 
holds for all $\eta\in\Rn.$
We refer to \cite{Kom} for the extensive analysis of these spaces and their
duals in $\Rn$.
However, such a local point of view does not tell us about the global
properties of $\phi$ such as its relation to the geometric or spectral 
properties of the group $G$, and this is the aim of this paper.
The characterisations that we give are global, 
i.e. they do not refer to the localisation of the spaces, but are expressed
in terms of the behaviour of the global Fourier transform and the
properties of the global Fourier coefficients.

Such global characterisations will be useful for applications. For example,
the Cauchy problem for the wave equation
\begin{equation}\label{WE}
\partial_{t}^{2}u-a(t)\L_{G}u=0
\end{equation}
is well-posed, in general, only in Gevrey spaces, if $a(t)$ becomes zero at some
points. However, in local coordinates \eqref{WE} becomes a second order equation
with space-dependent coefficients and lower order terms, the case when the
well-posedness results are, in general\footnote{The result of Bronshtein \cite{B}
holds but is, in general, not optimal for some types of equations or does not hold
for low regularity $a(t).$}, not available even on $\Rn$. 
At the same time,
in terms of the group Fourier transform the equation \eqref{WE} is basically
constant coefficients, and the global characterisation of Gevrey spaces together
with an energy inequality for \eqref{WE} yield the well-posedness result.
We will address this and other applications elsewhere, but we note that
in these problems both types of Gevrey spaces
appear naturally, see e.g. \cite{GR} for the Gevrey-Roumieu ultradifferentiable
and Gevrey-Beurling ultradistributional well-posedness of weakly
hyperbolic partial differential equations in the Euclidean space.

In Section \ref{SEC:Res} we will fix the notation and formulate our results.
We will also recall known (easy) characterisations for other spaces, such as
spaces of smooth functions, distributions, or Sobolev spaces over $L^{2}.$
The proof for the characterisation of Gevrey spaces will rely on the 
harmonic analysis on the group, the family of spaces $\ell^{p}(\Gh)$ on the
unitary dual introduced in \cite{RTb}, and to some extent on the analysis
of globally defined matrix-valued symbols of pseudo-differential operators
developed in \cite{RTb, RTi}. The analysis of ultradistributions will rely on
the theory of sequence spaces (echelon and co-echelon spaces), see e.g.
K\"othe \cite{Koe}, Ruckle \cite{WR}. 
Thus, we will first give characterisations of the so-called
$\alpha$-duals of the Gevrey spaces and then 
show that $\alpha$-duals
and topological duals coincide. We also prove that both Gevrey spaces
are perfect spaces, i.e. the $\alpha$-dual of its $\alpha$-dual is the
original space. This is done in Section \ref{SEC:alpha}, and the ultradistributions
are treated in Section \ref{SEC:ultra}.

We note that the case of the periodic
Gevrey spaces, which can be viewed as spaces on the torus
$\mathbb T^{n}$, has been characterised by the Fourier coefficients in
\cite{Ta}. However, that paper stopped short of characterising 
the topological duals (i.e. the corresponding ultradistributions), so already
in this case our characterisation in Theorem \ref{THM:duals}
appears to be new.

In the estimates throughout the paper the constants will be denoted by letter $C$
which may change value even in the same formula. If we want to emphasise the
change of the constant, we may use letters like $C', A_{1}$, etc.

\section{Results}
\label{SEC:Res}

We first fix the notation and recall known characterisations of several spaces.
We refer to \cite{RTb} for details on the following constructions.

Let $G$ be a compact Lie group of dimension $n$.
Let $\Gh$ denote the set of (equivalence classes of) continuous irreducible
unitary representations of $G$. 
Since $G$ is compact, $\Gh$ is discrete.
For $[\xi]\in\Gh$, by choosing a basis in the
representation space of $\xi$, we can view $\xi$ as a matrix-valued function
$\xi:G\to\C^{d_{\xi}\times d_{\xi}},$ where $d_{\xi}$ is the dimension of the
representation space of $\xi.$
For $f\in L^{1}(G)$ we define its global Fourier transform at $\xi$ by
$$
 \widehat{f}(\xi)=\int_{G} f(x) \xi(x)^{*} dx,
$$
where $dx$ is the normalised Haar measure on $G$. The Peter-Weyl theorem
implies the Fourier inversion formula
\begin{equation}\label{EQ:FS}
 f(x)=\sumxi d_{\xi} \Tr\p{\xi(x)\widehat{f}(\xi)}.
\end{equation}
For each $[\xi]\in\Gh$, the matrix elements of $\xi$ are the eigenfunctions for 
the Laplace-Beltrami operator $\L_{G}$
with the same eigenvalue which we denote by $-\lambda_{[\xi]}^{2}$, so that
$-\L_{G}\xi_{ij}(x)=\lambda_{[\xi]}^{2}\xi_{ij}(x),$ for all $1\leq i,j\leq d_{\xi}.$

Different spaces on the Lie group $G$ can be characterised in terms of comparing
the Fourier coefficients of functions with powers of the eigenvalues of the Laplace-Beltrami
operator. We denote $\jp{\xi}=(1+\lambda_{[\xi]}^{2})^{1/2}$, the eigenvalues of
the elliptic first-order pseudo-differential operator $(I-\L_{G})^{1/2}.$ 

Then, it is easy to see that
$f\in C^{\infty}(G)$ if and only if for every $M>0$ there exists $C>0$ such that
$\|\widehat{f}(\xi)\|_{\HS}\leq C\jp{\xi}^{-M},$ and
$u\in \D'(G)$ if and only if there exist $M>0$ and $C>0$ such that
$\|\widehat{u}(\xi)\|_{\HS}\leq C\jp{\xi}^{M},$ where we define
$\widehat{u}(\xi)_{ij}=u(\overline{\xi_{ji}})$, $1\leq i,j\leq d_{\xi}.$
For this and other occasions, we can write this as
$\widehat{u}(\xi)=u(\xi^{*})$ in the matrix notation. 
The appearance of the
Hilbert-Schmidt norm is natural in view of the Plancherel identity
$$
 (f,g)_{L^{2}(G)}=\sumxi d_{\xi} \Tr\p{\widehat{f}(\xi)\widehat{g}(\xi)^{*}},
$$
so that 
$$
 \|f\|_{L^{2}(G)}=\p{\sumxi d_{\xi} \|\widehat{f}(\xi)\|_{\HS}^{2}}^{1/2}=:
  \|\widehat{f}\|_{\ell^{2}(\Gh)}
$$
can be taken as the definition of the space $\ell^{2}(\Gh)$. Here, of course,
$ \|A\|_{\HS}=\sqrt{\Tr(A A^{*})}.$
It is convenient to use the sequence space
$$
\Sigma=\{\sigma=(\sigma(\xi))_{[\xi]\in\Gh}: \sigma(\xi)\in\C^{d_{\xi}\times d_{\xi}} \}.
$$
In \cite{RTb}, the authors introduced a family of spaces $\ell^{p}(\Gh)$,
$1\leq p<\infty$, by
saying that $\sigma\in\Sigma$ belongs to $\ell^{p}(\Gh)$ if the norm
$$
\|\sigma\|_{\ell^{p}(\Gh)}:=\p{\sumxi d_{\xi}^{p\p{\frac2p-\frac12}} \|\sigma(\xi)\|_{\HS}^{p}}^{1/p}
$$
if finite. There is also the space $\ell^{\infty}(\Gh)$ for which the norm
\begin{equation}\label{EQ:linf}
\|\sigma\|_{\ell^{\infty}(\Gh)}:=\sup_{[\xi]\in\Gh} d_{\xi}^{-\frac12} \|\sigma(\xi)\|_{\HS}
\end{equation}
is finite. These are interpolation spaces for which the Hausdorff-Young inequality holds,
in particular, we have
\begin{equation}\label{EQ:HY}
 \|\widehat{f}\|_{\ell^{\infty}(\Gh)}\leq \|f\|_{L^{1}(G)} \; \textrm{ and } \;
 \|\mathscr F^{-1}\sigma\|_{L^{\infty}(G)}\leq \|\sigma\|_{\ell^{1}(\Gh)},
\end{equation}
with $(\mathscr F^{-1}\sigma)(x)=\sumxi d_{\xi} \Tr\p{\xi(x)\sigma(\xi)}.$
We refer to \cite[Chapter 10]{RTb} for further details on these spaces.
Usual Sobolev spaces on $G$ as a manifold, defined by localisations, can be also 
characterised by the global condition
\begin{equation}\label{EQ:S1}
f\in H^{t}(G) \textrm{ if and only if } \jp{\xi}^t \widehat{f}(\xi)\in \ell^{2}(\Gh).
\end{equation}
For a multi-index $\alpha=(\alpha_{1},\ldots,\alpha_{n})$, we define
$|\alpha|=|\alpha_{1}|+\cdots+|\alpha_{n}|$ and
$\alpha!=\alpha_{1}!\cdots\alpha_{n}!.$
We will adopt the convention that $0!=1$ and $0^{0}=1.$

Let $X_{1},\ldots,X_{n}$ be a basis of the Lie algebra of $G$, normalised in some way, 
e.g. with respect to the Killing form. 
For a multi-index $\alpha=(\alpha_{1},\ldots,\alpha_{n})$, we define
the left-invariant differential operator of order $|\alpha|$,
$\partial^{\alpha}:=Y_{1}\cdots Y_{|\alpha|},$
with $Y_{j}\in\{X_{1},\cdots,X_{n}\}$, $1\leq j\leq |\alpha|$,
and $\sum_{j: Y_{j}=X_{k}} 1=\alpha_{k}$ for every $1\leq k\leq n.$
It means that $\paal$ is a composition of left-invariant
derivatives with respect to vectors $X_{1},\cdots,X_{n}$, such that each
$X_{k}$ enters $\paal$ exactly $\alpha_{k}$ times.
There is a small abuse of notation here since we do not specify in the
notation $\paal$ the order
of vectors $X_{1},\cdots,X_{n}$ entering in $\paal$, but this will not be important
for the arguments in the paper. The reason we define $\paal$ in this way is to
take care of the non-commutativity of left-invariant differential
operators corresponding to the vector fields $X_{k}.$

We will distinguish between two families of Sobolev spaces over $L^{2}$.
The first one is defined by 
$H^{t}(G)=\left\{f\in L^2 (G) : (I-\L_G)^{t/2} f\in L^{2}(G) \right\}$ with the norm
\begin{equation}\label{EQ:S2}
\|f\|_{H^{t}(G)}:=\| (I-\L_G)^{t/2} f\|_{L^{2}(G)}=\|\jp{\xi}^{t}\widehat{f}(\xi)\|_{\ell^{2}(\Gh)}.
\end{equation}
The second one is defined for $k\in \N_{0}\equiv \N\cup\{0\}$ by
$$W^{k,2}=\left\{f\in L^2(G): 
\|f\|_{W^{k,2}}:=\sum_{|\alpha|\leq k}||\partial^{\alpha}f||_{L^{2}(G)}<\infty\right\}.$$
Obviously, $H^{k}\simeq W^{k,2}$ for any $k\in\N_{0}$ but for us the relation between
norms will be of importance, especially as $k$ will tend to infinity. \\

Let $0<s<\infty.$ We first fix the notation for the Gevrey 
spaces and then formulate the results. In the definitions below we allow any $s>0$, and the
characterisation of $\alpha$-duals in the sequel will still hold. However, when dealing with
ultradistributions we will be restricting to $s\geq 1$.

\begin{defi} \label{DEF:GR}
Gevrey-Roumieu(R) class $\gamma_{s}(G)$ is the space of 
functions $\phi\in{C^{\infty}(G)}$ for which there exist constants $A>0$ and $C>0$ 
such that for all multi-indices $\alpha,$ we have
\begin{equation}\label{EQ:GR}
||\partial^{\alpha}\phi||_{L^\infty}\equiv\sup_{x\in G}{|\partial^{\alpha}\phi(x)|\leq 
C A^{|\alpha|}\left(\alpha !\right)^{s}}.
\end{equation}
Functions $\phi\in{\gamma_{s}}(G)$ are called ultradifferentiable functions of Gevrey-Roumieu 
class of order s.
\end{defi}
For $s=1$ we obtain the space of analytic functions, and for $s>1$ the space of Gevrey-Roumieu
functions on $G$ considered as a manifold, by saying that the function is in the Gevrey-Roumieu
class locally in every coordinate chart. The same is true for the other Gevrey space:

\begin{defi} Gevrey-Beurling(B) class $\gamma_{(s)}(G)$ is the space of functions 
$\phi\in{C^{\infty}(G)}$ such that for every $A>0$ there exists $C_A>0$ so that 
for all multi-indices $\alpha,$ we have
$$
||\partial^{\alpha}\phi||_{L^\infty}\equiv\sup_{x\in G}
{|\partial^{\alpha}f(x)|\leq C_A A^{|\alpha|}\left(\alpha !\right)^{s}}.
$$
Functions $\phi\in{\gamma_{(s)}}(G)$ are called ultradifferentiable functions 
of Gevrey-Beurling class of order s.
\end{defi}

\begin{thm}\label{THM:Gevrey}
Let $0<s<\infty$. \\
{\rm \bf (R)} We have $\phi\in \gamma_{s}(G)$ if and only if there exist $B>0$ and $K>0$
such that
\begin{equation}\label{EQ:GR}
||\widehat{\phi}(\xi)||_{\HS}\leq K e^{-B\jp{\xi}^{1/s}}
\end{equation}
holds for all $[\xi]\in \Gh.$\\
{\rm\bf (B)} We have $\phi\in \gamma_{(s)}(G)$ if and only if for every $B>0$ there exists $K_B>0$
such that 
\begin{equation}\label{EQ:GB}
||\widehat{\phi}(\xi)||_{\HS}\leq K_B e^{-B\jp{\xi}^{1/s}}
\end{equation}
holds for all $[\xi]\in \Gh.$
\end{thm}

Expressions appearing in the definitions can be taken as seminorms, and the spaces
are equipped with  the inductive and projective topologies, respectively\footnote{See
also Definition \ref{DEF:dual} for an equivalent formulation.}.
We now turn to ultradistributions.

\begin{defi} 
The space of continuous linear functionals on 
$\gamma_s(G)\left(\textrm{or}~\gamma_{(s)}(G)\right)$ 
is called the space of ultradistributions and is denoted 
by $\gamma_s'(G)\left(\textrm{or}~\gamma_{(s)}'(G)\right),$
respectively.
\end{defi}

For any $v\in \gamma_{s}'(G)\left(\textrm{or}~\gamma_{(s)}'(G)\right)$, 
for $[\xi]\in\Gh$, we define the Fourier coefficients 
$\widehat{v}(\xi):=\jp{v,\xi^{\ast}}\equiv v(\xi^{*}).$
These are well-defined since $G$ is compact and hence $\xi(x)$ are actually
analytic.

\begin{thm}\label{THM:duals}
Let $1\leq s<\infty$. \\
{\rm\bf (R)} We have $v \in \gamma_s'(G)$ 
if and only if for every $B>0$ there exists $K_B>0$ such that
\begin{equation}\label{EQ:ade1}
\|\widehat{v}(\xi)\|_{\HS}  \leq K_B e^{B\left\langle \xi  \right\rangle ^{\frac{1} {s}} }
\end{equation}
holds for all $ [\xi] \in \Gh$. \\
{\rm\bf (B)} We have $v \in \gamma_{(s)}'(G)$ 
if and only if there exist $B>0$ and $K_{B}>0$ such that
\eqref{EQ:ade1}
holds for all $ [\xi] \in \Gh$.
\end{thm}
The proof of Theorem \ref{THM:duals}
follows from the characterisation of $\alpha$-duals of\footnote{The characterisation of
$\alpha$-duals is valid for all $0<s<\infty$.}  the
Gevrey spaces in Theorem \ref{THM:aduals}
and the equivalence of the topological duals and $\alpha$-duals 
in Theorem \ref{THM:equiv}.

\medskip

The result on groups implies the corresponding characterisation on
compact homogeneous spaces $M$. First we fix the notation.
Let $G$ be a compact motion group of $M$ and let $H$ be the stationary
subgroup of some point.
Alternatively, we can start with a compact Lie group $G$ with
a closed subgroup $H$.  
The homogeneous space $M=G/H$ is an analytic manifold in a canonical way
(see, for example, \cite{Br} or \cite{St} as textbooks on this subject).
We normalise measures so
that the measure on $H$ is a probability one.
Typical examples are the spheres ${\mathbb S}^{n}={\rm SO}(n+1)/{\rm SO}(n)$
or complex spheres
(complex projective spaces) $\mathbb P\mathbb C^{n}={\rm SU}(n+1)/{\rm SU}(n)$.

We denote by $\Gh_{0}$ the subset of $\Gh$ of representations that are
class I with respect to the subgroup  $H$. This means that 
$[\xi]\in\Gh_{0}$ if $\xi$ has at least one non-zero invariant vector $a$ with respect
to $H$, i.e. that
$\xi(h)a=a$ for all $h\in H.$ 
Let $\Hcal_{\xi}$ denote the representation space of $\xi(x):\Hcal_{\xi}\to\Hcal_{\xi}$ 
and let $\B_{\xi}$ be the space of these invariant vectors.
Let $k_{\xi}=\dim\B_{\xi}.$
We fix an orthonormal basis of $\Hcal_{\xi}$ so that
its first $k_{\xi}$ vectors are the basis of $B_{\xi}.$
The matrix elements $\xi_{ij}(x)$, $1\leq j\leq k_{\xi}$,
are invariant under the right shifts by $H$.
We refer to \cite{Vi} for the details of these constructions.

We can identify Gevrey functions on $M=G/H$ with Gevrey functions on 
$G$ which are constant on left cosets with respect to $H$. 
Here we will restrict to $s\geq 1$ to see the equivalence of spaces
using their localisation. This identification gives rise to the 
corresponding identification of ultradistributions.
Thus, for a function $f\in \gamma_{s}(M)$ we can
recover it by the Fourier series of its canonical
lifting $\wt{f}(g):=f(gH)$ to $G$,
$\wt{f}\in \gamma_{s}(G)$, and the Fourier
coefficients satisfy  $\widehat{\wt{f}}(\xi)=0$ for all representations
with $[\xi]\not\in\Gh_{0}$.
Also, for class I representations $[\xi]\in\Gh_{0}$ we have 
$\widehat{\wt{f}}(\xi)_{{ij}}=0$ for $i>k_{\xi}$.

With this, we can write the Fourier series of $f$ (or of $\wt{f}$, but as
we said, from now on we will identify these and denote both by $f$) in terms of
the spherical functions $\xi_{ij}$ of the representations
$\xi$, $[\xi]\in\Gh_{0}$, with respect to the subgroup
$H$. 
Namely, the Fourier series \eqref{EQ:FS} becomes
\begin{equation}\label{EQ:FSh}
f(x)=\sum_{[\xi]\in\Gh_{0}} d_{\xi} \sum_{i=1}^{d_{\xi}}\sum_{j=1}^{k_{\xi}}
\widehat{f}(\xi)_{ji}\xi_{ij}(x).
\end{equation}
In view of this,
we will say that the collection of Fourier coefficients
$\{\widehat{\phi}(\xi)_{ij}: [\xi]\in\Gh, 1\leq i,j\leq d_{\xi}\}$ is 
of class I with respect to $H$ if
$\widehat{\phi}(\xi)_{ij}=0$ whenever $[\xi]\not\in\Gh_{0}$ or
$i>k_{\xi}.$ By the above discussion, if the collection of Fourier
coefficients is of class I with respect to $H$, then the expressions
\eqref{EQ:FS} and \eqref{EQ:FSh} coincide and yield a function
$f$ such that $f(xh)=f(h)$ for all $h\in H$, so that this function becomes
a function on the homogeneous space $G/H$. The same applies
to (ultra)distributions with the standard distributional interpretation.
With these identifications, Theorem \ref{THM:Gevrey} immediately implies

\begin{thm}\label{THM:Gevreyh}
Let $1\leq s<\infty$. \\
{\rm \bf (R)} We have $\phi\in \gamma_{s}(G/H)$ if and only if 
its Fourier coefficients are of class I with respect to $H$ and, moreover,
there exist $B>0$ and $K>0$
such that
\begin{equation}\label{EQ:GRh}
||\widehat{\phi}(\xi)||_{\HS}\leq K e^{-B\jp{\xi}^{1/s}}
\end{equation}
holds for all $[\xi]\in \Gh_{0}.$\\
{\rm\bf (B)} We have $\phi\in \gamma_{(s)}(G)$ if and only if 
its Fourier coefficients are of class I with respect to $H$ and, moreover,
for every $B>0$ there exists $K_B>0$ such that 
\begin{equation}\label{EQ:GBh}
||\widehat{\phi}(\xi)||_{\HS}\leq K_B e^{-B\jp{\xi}^{1/s}}
\end{equation}
holds for all $[\xi]\in \Gh_{0}.$
\end{thm}

It would be possible to extend Theorem \ref{THM:Gevreyh} to the range
$0<s<\infty$ by adopting Definition \ref{DEF:GR} starting with a frame of vector
fields on $M$, but instead of obtaining the result immediately from
Theorem \ref{THM:Gevrey} we would have to go again through arguments similar to
those used to prove Theorem \ref{THM:Gevrey}. Since we are interested in characterising 
the standard invariantly defined
Gevrey spaces we decided not to lengthen the proof in this way.
On the other hand, it is also possible to prove the characterisations on homogeneous
spaces $G/H$ first and then obtain those on the group $G$ by taking $H$ to be trivial.
However, some steps would become more technical since we would have to
deal with frames of vector fields instead of the basis of left-invariant
vector fields on $G$, and elements of the symbolic calculus used in the proof would
become more complicated.

We also have the ultradistributional result following from Theorem
\ref{THM:duals}.

\begin{thm}\label{THM:dualsh}
Let $1\leq s<\infty$. \\
{\rm\bf (R)} We have $v \in \gamma_s'(G/H)$ 
if and only if 
its Fourier coefficients are of class I with respect to $H$ and, moreover,
for every $B>0$ there exists $K_B>0$ such that
\begin{equation}\label{EQ:ade1h}
\|\widehat{v}(\xi)\|_{\HS}  \leq K_B e^{B\left\langle \xi  \right\rangle ^{\frac{1} {s}} }
\end{equation}
holds for all $ [\xi] \in \Gh_{0}$. \\
{\rm\bf (B)} We have $v \in \gamma_{(s)}'(G/H)$ 
if and only if 
its Fourier coefficients are of class I with respect to $H$ and, moreover,
there exist $B>0$ and $K_{B}>0$ such that
\eqref{EQ:ade1h}
holds for all $[\xi] \in\Gh_{0}$.
\end{thm}

Finally, we remark that in the harmonic analysis on compact Lie groups sometimes
another version of $\ell^{p}(\Gh)$ spaces appears using Schatten $p$-norms.
However, in the context of Gevrey spaces and ultradistributions eventual results hold
for all such norms. Indeed, given our results with the Hilbert-Schmidt norm,
by an argument similar to that of Lemma \ref{L:c} below,
we can put any Schatten norm $\|\cdot\|_{S_{p}}$, $1\leq p\leq\infty,$ 
instead of the Hilbert-Schmidt norm $\|\cdot\|_{\HS}$ in any of
our characterisations and they still continue to hold.

\section{Gevrey classes on compact Lie groups}

We will need two relations between dimensions of representations and the
eigenvalues of the Laplace-Beltrami operator. 
On one hand, it follows from the
Weyl character formula that
\begin{equation}\label{EQ:dxi}
d_{\xi}\leq C\jp{\xi}^{\frac{n-{\rm rank}G}{2}}\leq C\jp{\xi}^{\frac{n}{2}},
\end{equation}
with the latter\footnote{Namely, the inequality $d_{\xi}\leq   C\jp{\xi}^{\frac{n}{2}}$.}
also following directly from the Weyl 
asymptotic formula for the eigenvalue
counting function for $\L_{G}$, see e.g. \cite[Prop. 10.3.19]{RTb}. 
This implies, in particular, that for any $0\leq p<\infty$ and 
any $s>0$ and $B>0$ we have
\begin{equation}\label{EQ:exp}
 \sup_{[\xi]\in\Gh} d_{\xi}^p e^{-B\jp{\xi}^{1/s}}<\infty.
\end{equation}
On the other hand, the following
convergence for the series will be useful for us:

\begin{lemma} \label{L:series}
We have 
$\sumxi \ d_{\xi}^{2}\ \jp{\xi}^{-2t}<\infty$ if and only if $t>\frac{n}{2}.$
\end{lemma}
\begin{proof}
We notice that for the $\delta$-distribution at the unit element of the group, 
$\widehat{\delta}(\xi)=I_{d_{\xi}}$
is the identity matrix of size $d_{\xi}\times d_{\xi}$. Hence, in view of
\eqref{EQ:S1} and \eqref{EQ:S2}, we can write
$$
\sumxi d_{\xi}^{2} \jp{\xi}^{-2t}=
\sumxi d_{\xi} \jp{\xi}^{-2t}\|\widehat{\delta}(\xi)\|_{\HS}^{2}=
\|(I-\L_{G})^{-t/2}\delta\|_{L^{2}(G)}^{2}=\|\delta\|_{H^{-t}(G)}^{2}.
$$
By using the localisation of $H^{-t}(G)$ this is finite if and only if $t>n/2.$
\end{proof}

We denote by $\Ghs$ the set of representations from $\Gh$ excluding the
trivial representation. For $[\xi]\in\Gh$, we denote 
$|\xi|:=\lambda_{\xi}\geq 0$, the eigenvalue of the operator
$(-\L_{G})^{1/2}$ corresponding to the representation $\xi.$
For $[\xi]\in\Ghs$ we have $|\xi|>0$ (see e.g. \cite{F}),
and for $[\xi]\in\Gh\backslash\Ghs$ we have $|\xi|=0.$
From the definition, we have $|\xi|\leq \jp{\xi}.$ On the other hand,
let $\lambda_{1}^{2}>0$ be the smallest positive eigenvalue of $-\L_{G}.$
Then, for $[\xi]\in\Ghs$ we have $\lambda_{\xi}\geq \lambda_{1}$, implying
$$
1+\lambda_{\xi}^{2}\leq \p{\frac{1}{\lambda_{1}^{2}}+1}\lambda_{\xi}^{2},
$$
so that altogether we record the inequality
\begin{equation}\label{EQ:ll}
|\xi|\leq \jp{\xi}\leq \p{1+\frac{1}{\lambda_{1}^{2}}}^{1/2}|\xi|,
\quad \textrm{ for all } [\xi]\in\Ghs.
\end{equation}

We will need the following simple lemma which we prove for completeness.
Let $a\in\C^{d\times d}$ be a matrix, and for $1\leq p<\infty$ we 
denote by $\ell^p(\C)$ the space of such matrices with the norm
$$
 \|a\|_{\ell^p(\C)}=\p{\sum_{i,j=1}^d |a_{ij}|^p}^{1/p},
$$
and for $p=\infty$, 
$ \|a\|_{\ell^\infty(\C)}=\sup_{1\leq i,j\leq d} |a_{ij}|.$
We note that $\|a\|_{\ell^2(\C)}=\|a\|_{\HS}.$
We adopt the usual convention $\frac{c}{\infty}=0$ for any $c\in\mathbb R.$
\begin{lemma} \label{L:c}
Let $1\leq p< q\leq\infty$ and let $a\in\C^{d\times d}.$ Then we have
\begin{equation}\label{EQ:in}
 \|a\|_{\ell^p(\C)}\leq d^{2\p{\frac1p-\frac1q}}\|a\|_{\ell^q(\C)}
 \quad\textrm{ and } \quad
 \|a\|_{\ell^q(\C)}\leq d^{\frac{2}{q}}\|a\|_{\ell^p(\C)}.
\end{equation} 
\end{lemma}
\begin{proof}
For $q<\infty$, we apply H\"older's inequality with $r=\frac{q}{p}$ and
$r'=\frac{q}{q-p}$ to get
$$
 \|a\|_{\ell^p(\C)}^p=\sum_{i,j=1}^d |a_{ij}|^p\leq
 \p{\sum_{i,j=1}^d |a_{ij}|^{pr}}^{1/r}\p{\sum_{i,j=1}^d 1}^{1/r'}=
 \|a\|_{\ell^q(\C)}^{p} d^{2\frac{q-p}{q}},
$$
implying \eqref{EQ:in} for this range.
Conversely, we have
$$
 \|a\|_{\ell^q(\C)}^q=\sum_{i,j=1}^d |a_{ij}|^q\leq
 \sum_{i,j=1}^d \|a\|_{\ell^p(\C)}^q= d^2\|a\|_{\ell^p(\C)}^q,
$$
proving the other part of \eqref{EQ:in} for this range.
For $q=\infty$, we have
$ \|a\|_{\ell^p(\C)}\leq \p{\sum_{i,j=1}^d  \|a\|^{p}_{\ell^\infty(\C)}}^{1/p}\leq
\|a\|_{\ell^\infty(\C)} d^{2/p}.$ 
Conversely, we have trivially $\|a\|_{\ell^\infty(\C)}\leq \|a\|_{\ell^p(\C)},$
completing the proof.
\end{proof} 

We observe that the Gevrey spaces can be described in terms of $L^{2}$-norms,
and this will be useful to us in the sequel.
\begin{lemma}\label{L:gl2}
We have $\phi\in \gamma_{s}(G)$ if and only if there exist constants $A>0$ and $C>0$
such that for all multi-indices $\alpha$ we have
\begin{equation}\label{EQ:gl2}
\|\partial^{\alpha}\phi\|_{L^2}\leq 
C A^{|\alpha|}\left(\alpha !\right)^{s}.
\end{equation}
We also have $\phi\in \gamma_{(s)}(G)$ if and only if for every $A>0$ there exists $C_{A}>0$
such that for all multi-indices $\alpha$ we have
$$
\|\partial^{\alpha}\phi\|_{L^2}\leq 
C_{A} A^{|\alpha|}\left(\alpha !\right)^{s}.
$$
\end{lemma}
\begin{proof} 
We prove the Gevrey-Roumieu case (R) as the Gevrey-Beurling case (B) is similar.
For $\phi\in \gamma_{s}(G),$ \eqref{EQ:gl2} follows in view of the continuous embedding 
$L^{\infty}(G)\subset L^{2}(G)$ with 
$\|f\|_{L^{2}}\leq \|f\|_{L^{\infty}}$ since the measure is normalised.

Now suppose that 
for $\phi\in C^{\infty}(G)$ we have \eqref{EQ:gl2}.
In view of \eqref{EQ:HY}, and using Lemma \ref{L:series} with 
an integer $k>n/2$, we obtain\footnote{Note that this can be adopted to give a simple proof 
of the Sobolev embedding theorem.}
\begin{eqnarray} 
\|\phi\|_{L^{\infty}}&\leq&
\sumxi d_{\xi}^{3/2}\|\widehat{\phi}(\xi)\|_{\HS}\nonumber\\
&\leq& 
\left(\sumxi d_{\xi} \|\widehat{\phi}(\xi)\|^{2}_{\HS}\jp{\xi}^{2k}\right)^{1/2}
\left(\sumxi d_{\xi}^{2}\jp{\xi}^{-2k}\right)^{1/2}\nonumber\\
&\le & C\|(I-\L_G)^{k/2}\phi\|_{L^2}\nonumber\\
&\leq& C_k\sum_{|\beta|\leq k}\|\partial^{\beta}\phi\|_{L^2},\nonumber 
\end{eqnarray}
with constant $C_{k}$ depending only on $G$.
Consequently we also have
\begin{equation}\label{EQ:aux1}
\|\partial^{\alpha}\phi\|_{L^\infty} \leq C_{k}\sum_{|\beta|\leq k}\|{\partial^{\alpha+\beta}\phi}\|_{L^{2}}.
\end{equation}
Using the inequalities
\begin{equation}\label{ineq}
\alpha !\leq |\alpha|!,\quad |\alpha|!\leq n^{|\alpha|}\alpha!
\quad\textrm{ and } \quad(|\alpha|+k)!\leq 2^{|\alpha|+k}k !|\alpha|!,
\end{equation}
in view of \eqref{EQ:aux1} and \eqref{EQ:gl2} 
we get
\begin{eqnarray}
||\partial^{\alpha}\phi||_{L^\infty}&\leq& C_{k} 
A^{|\alpha|+k}\sum_{|\beta|\leq k}\left((\alpha +\beta)!\right)^{s}\nonumber\\
&\leq& C_{k} A^{|\alpha|+k}\sum_{|\beta|\leq k}\left((|\alpha| + k)!\right)^{s}\nonumber\\
&\leq& C_k' A^{|\alpha|+k}  ( 2^{|\alpha|+k}k!)^s(|\alpha|!)^s \nonumber\\
&\leq& C_{k}'' A_{1}^{|\alpha|}(n^{|\alpha|} \alpha!)^s\nonumber\\
&\leq& C_{k}'' A_{2}^{|\alpha|}(\alpha!)^s,\nonumber
 \end{eqnarray}
with constants $C_{k}''$ and $A_{2}$ independent of $\alpha$, implying that
$\phi\in\gamma_{s}(G)$ and completing the proof.
\end{proof}
 
The following proposition prepares the possibility to passing to the conditions formulated on
the Fourier transform side.
\begin{prop} \label{PROP:l}
We have $\phi\in{\gamma_{s}(G)}$ if and only if
there exist constants $A>0$ and $C>0$ such that 
\begin{equation}\label{EQ:clap}
||\left(-\L_G\right)^{k}\phi||_{L^\infty}\leq C A^{2k}\left((2k)!\right)^s
\end{equation} 
holds for all $k\in\N_{0}$.  Also, $\phi\in\gamma_{(s)}(G)$ if and only if for every $A>0$ there
exists $C_{A}>0$ such that for all $k\in\N_{0}$ we have
$$
||\left(-\L_G\right)^{k}\phi||_{L^\infty}\leq C_A A^{2k}\left((2k)!\right)^s.
$$
\end{prop}

\begin{proof} 
We prove the Gevrey-Roumieu case \eqref{EQ:clap} and indicate small
additions to the argument for $\gamma_{(s)}(G)$.
Thus, let $\phi\in{\gamma_{s}(G)}$. Recall that by the definition there exist
some $A>0,$ $C>0$ such that for all multi-indices $\alpha$ we have
$$
||\partial^{\alpha}\phi||_{L^\infty}=
\sup_{x\in G}{|\partial^{\alpha}\phi(x)|\leq C
A^{|\alpha|}\left(\alpha !\right)^{s}}.$$ 
We will use the fact that for the compact Lie group $G$ the Laplace-Beltrami
operator $\L_{G}$ is given by 
$\L_G=X_1^2+X_2^2+...+X_n^2$, where $X_i$, $i=1,2,\ldots,n$, is a set of left-invariant 
vector fields corresponding to a normalised basis of the Lie algebra of $G$.
Then by the multinomial theorem\footnote{The form in which we use it is 
adapted to non-commutativity of vector fields. Namely, although the coefficients are
all equal to one in the non-commutative form, the multinomial coefficient appears once
we make a choice for $\alpha=(\alpha_{1},\cdots,\alpha_{n}).$}  
and using \eqref{ineq}, with
$Y_{j}\in\{X_{1},\ldots,X_{n}\}$, $1\leq j\leq |\alpha|$,
we can estimate
\begin{eqnarray}\label{EQ:estLG}
|(-{\mathcal{L}}_G)^{k}\phi(x)|&\leq& C\sum_{|\alpha|= k} \frac{k!}{\alpha!}\left| Y_1^{2}\ldots
Y_{|\alpha|}^{2}\phi(x)\right|\nonumber\\
&\leq& C \sum_{|\alpha|= k}\frac{k!}{\alpha!}[(2|\alpha|)!]^{s}A^{2|\alpha|}\nonumber\\
&\leq& C A^{2k}[(2k)!]^{s}
\sum_{|\alpha|= k}\frac{k! n^{|\alpha|}}{|\alpha|!} \nonumber\\
&\leq& C_{1} A^{2k}[(2k)!]^{s} n^{k} k^{n-1} \nonumber\\
&\leq& C_2 A_{1}^{2k}[(2k)!]^{s},
\end{eqnarray}
with $A_{1}=2nA$,
implying \eqref{EQ:clap}. For the Gevrey-Beurling case $\gamma_{(s)}(G)$, we observe that
we can obtain any $A_{1}>0$ in \eqref{EQ:estLG}
by using $A=\frac{A_{1}}{2n}$ in the Gevrey estimates for $\phi\in\gamma_{(s)}(G).$

Conversely, suppose $\phi\in C^{\infty}(G)$ is such that the inequalities \eqref{EQ:clap} hold. 
First we note that for $|\alpha|=0$ the estimate \eqref{EQ:GR} follows from
\eqref{EQ:clap} with $k=0$, so that we can assume $|\alpha|>0.$

Following \cite{RTi}, we define the symbol of $\paal$ to be
$\sigma_{\paal}(\xi)=\xi(x)^{*}\paal\xi(x),$ and we have 
$\sigma_{\paal}(\xi)\in\C^{d_{\xi}\times d_{\xi}}$ is
independent of $x$ since $\paal$ is left-invariant. For the in-depth analysis of symbols
and symbolic calculus for general operators on $G$ we refer to \cite{RTb, RTi} but we
will use only basic things here. In particular, we have 
$$
\paal\phi(x)=\sumxi d_{\xi} \Tr\p{\xi(x)\sigma_{\paal}(\xi)\widehat\phi(\xi)}.
$$
First we calculate the operator norm $||\sigma_\paal(\xi)||_{op}$ of the matrix multiplication
by $\sigma_\paal(\xi)$.
Since $\partial^{\alpha}=Y_{1}\cdots Y_{|\alpha|}$ and
$Y_{j}\in\{X_{1},\ldots,X_{n}\}$ are all left-invariant,
we have $\sigma_{\paal}=\sigma_{Y_1}\cdots \sigma_{Y_{|\alpha|}}$, so that
we get 
$$
\|\sigma_{\paal}(\xi)\|_{op}\leq
\|\sigma_{X_1}(\xi)\|^{\alpha_1}_{op}\cdots \|\sigma_{X_n}(\xi)\|^{\alpha_n}_{op}.
$$
Now, since $X_{j}$ are operators of the first order, one can show
(see e.g. \cite[Lemma 8.6]{RTi},  or \cite[Section 10.9.1]{RTb} for general arguments) that 
$||\sigma_{X_j}(\xi)||_{op}\leq C_{j}\jp{\xi}$ for some constants $C_{j},$ 
$j=1,\ldots,n$. 
Let $C_0=\sup_{j}C_j+1,$ then we have 
\begin{equation}\label{EQ:paalnorm}
\|\sigma_\paal(\xi)\|_{op}\leq C_{0}^{|\alpha|}\jp{\xi}^{|\alpha|}.
\end{equation}
Let us define $\sigma_{P_{\alpha}}\in\Sigma$ by
setting $\sigma_{P_{\alpha}}(\xi):=|\xi|^{-2k}\sigma_{\paal}(\xi)$ for 
$[\xi]\in\Ghs$, and by $\sigma_{P_{\alpha}}(\xi):=0$ for 
$[\xi]\in\Gh\backslash\Ghs.$ 
This gives the corresponding operator
\begin{eqnarray} \label{EQ:Pal}
(P_{\alpha}\phi)(x) & = & \sumxi d_{\xi} \Tr\p{\xi(x)\sigma_{P_{\alpha}}(\xi)\widehat\phi(\xi)} .
\end{eqnarray}
From \eqref{EQ:paalnorm} we obtain
\begin{equation}\label{EQ:paalnorm2}
\|\sigma_{P_{\alpha}}(\xi)\|_{op}\leq C_{0}^{|\alpha|} \jp{\xi}^{|\alpha|} |\xi|^{-2k} 
\textrm{ for all } [\xi]\in\Ghs.
\end{equation}
Now, for $[\xi]\in \Ghs$, from \eqref{EQ:ll} we have
$$
|\xi|^{-2k} \leq C_{1}^{2k}\jp{\xi}^{-2k},\quad
C_{1}=\p{1+\frac{1}{\lambda_{1}^{2}}}^{1/2}.
$$
Together with \eqref{EQ:paalnorm2}, and the trivial estimate for $[\xi]\in\Gh\backslash\Ghs$,
we obtain
\begin{equation}\label{EQ:paalnorm3}
\|\sigma_{P_{\alpha}}(\xi)\|_{op}\leq C_{0}^{|\alpha|} C_{1}^{2k} \jp{\xi}^{|\alpha|-2k} 
\textrm{ for all } [\xi]\in\Gh.
\end{equation}
Using \eqref{EQ:Pal} and the Plancherel identity, we estimate
\begin{eqnarray*} 
|P_{\alpha}\phi(x)| &\leq& \sumxi d_{\xi} \|\xi(x)\sigma_{P_{\alpha}}(\xi)\|_{\HS}
\|\widehat{\phi}(\xi)\|_{\HS}\nonumber\\
&\leq& \left(\sumxi d_{\xi}\|\widehat{\phi}(\xi)\|^{2}_{\HS}\right)^{1/2}
\left(\sumxi d_{\xi} 
\|\sigma_{P_{\alpha}}(\xi)\|_{op}^{2}
\|\xi(x)\|_{\HS}^{2}
\right)^{1/2}\nonumber\\
&=& \|\phi\|_{L^{2}}
\left(\sumxi d_{\xi}^{2} \|\sigma_{P_{\alpha}}(\xi)\|_{op}^{2}\right)^{1/2}.\nonumber
\end{eqnarray*} 
From this and \eqref{EQ:paalnorm3} we conclude that
\begin{eqnarray*} 
|P_{\alpha}\phi(x)| 
\leq \|\phi\|_{L^{2}} 
C_{0}^{|\alpha|} C_{1}^{2k}
\left(\sumxi d_{\xi}^{2}\jp{\xi}^{-2(2k-|\alpha|)}\right)^{1/2}.
\end{eqnarray*} 
Now, in view of Lemma \ref{L:series} the series on the right hand side converges
provided that $2k-|\alpha|>n/2.$
Therefore, for $2k-|\alpha|>n/2$ we obtain
\begin{equation}\label{EQ:pal2}
\|P_{\alpha}\phi\|_{L^2}\leq C C_{2}^{2k} \|\phi\|_{L^{2}},
\end{equation}
with some $C$ and $C_{2}=C_{0}C_{1}$ independent of $k$ and $\alpha$.
We note that here we used that $|\alpha|\leq 2k$ and that we can
always have $C_{0}\geq 1$.

We now observe that from the definition of $\sigma_{P_{\alpha}}$ we have 
\begin{equation}\label{EQ:sp}
\sigma_{\paal}(\xi)=\sigma_{P_{\alpha}}(\xi) |\xi|^{2k}
\end{equation}
for all $[\xi]\in\Ghs$.
On the other hand, since we assumed $|\alpha|\not=0$, for
$[\xi]\in\Gh\backslash\Ghs$ we have
$\sigma_{\paal}(\xi)=\xi(x)^{*}\paal\xi(x)=0$, so that
\eqref{EQ:sp} holds true for all $[\xi]\in\Gh.$
This implies that in the operator sense, we have
$\paal=P_{\alpha}\circ (-\L_{G})^{k}.$
Therefore, from this relation and \eqref{EQ:pal2}, for $|\alpha|<2k-n/2$, we get
\begin{eqnarray*}
\|\partial^{\alpha}\phi\|_{L^2}^{2}&=& \|P_{\alpha}\circ(-\L_{G})^{k}\phi\|_{L^{2}}^{2}\nonumber\\
&\leq& C C_{2}^{4k} \int_{G} |(-\L_{G})^k \phi(x)|^2 dx\nonumber\\
&\leq& C' C_2^{4k} A^{4k} ((2k)!)^{2s}\nonumber\\
&\leq& C' A_1^{4k} ((2k)!)^{2s},
\end{eqnarray*}
where we have used the assumption \eqref{EQ:clap}, and with $C'$ and 
$A_{1}=C_{2}A$ independent
of $k$ and $\alpha$.
Hence we have $\|\partial^{\alpha}\phi\|_{L^2}\leq C A_1^{2k}((2k)!)^s$ for all $|\alpha|< 2k-n/2.$
Then, for every $\beta$, by the above argument, taking an integer $k$ such that
$|\beta|+4n\geq 2k>|\beta|+n/2$,  if $A_{1}\geq 1$, we obtain 
\begin{eqnarray*}
\|\partial^{\beta}\phi\|_{L^2}\leq C A_1^{|\beta|+4n}\left((|\beta|+4n)!\right)^{s}
\leq C' A_1^{|\beta|} \left(2^{|\beta|+4n}(4n)!|\beta|!\right)^{s}
\leq C''A_{2}^{|\beta|} (\beta!)^{s},
\end{eqnarray*} 
in view of inequalities \eqref{ineq}.
By Lemma  \ref{L:gl2} it follows that $\phi\in\gamma_{s}(G).$

If $A_{1}<1$ (in the case of  $\gamma_{(s)}(G)$), we estimate
$$
\|\partial^{\beta}\phi\|_{L^2}\leq C A_1^{|\beta|+n/2}\left((|\beta|+4n)!\right)^{s}
\leq C''A_{3}^{|\beta|} (\beta!)^{s}
$$
by a similar argument.
The relation between constants, namely $A_{1}=C_{2}A$ and
$A_{3}=2nA_{1}$, implies that the case of
$\gamma_{(s)}(G)$ also holds true.
\end{proof}

We can now pass to the Fourier transform side.

\begin{lemma} \label{L:FT}
For $\phi\in \gamma_{s}(G)$, there exist constants $C>0$ and $A>0$ such that
\begin{equation}\label{EQ:FT1}
||\widehat{\phi}(\xi)||_{\HS}\leq C d_{\xi}^{1/2} |\xi|^{-2m} A^{2m}\left((2m)!\right)^{s}
\end{equation}
holds for all $m\in\N_{0}$ and $[\xi]\in\Ghs.$ 
Also,
for $\phi\in \gamma_{(s)}(G)$, for every $A>0$ there exists $C_{A}>0$ such that
$$
||\widehat{\phi}(\xi)||_{\HS}\leq C_{A} d_{\xi}^{1/2} |\xi|^{-2m} A^{2m}\left((2m)!\right)^{s}$$ 
holds for all $m\in\N_{0}$ and $[\xi]\in\Ghs.$
\end{lemma}

\begin{proof} 
We will treat the case $\gamma_{s}$ since $\gamma_{(s)}$ is analogous.
Using the fact that the Fourier transform is a bounded linear operator 
from $L^{1}(G)$ to $l^{\infty}(\widehat{G})$, see \eqref{EQ:HY},
and using Proposition \ref{PROP:l}, we obtain
\begin{eqnarray*}
|||\xi|^{2m}\widehat{\phi}(\xi)||_{l^{\infty}(\Gh)}
&\leq&\int_{G}|\left(-{\mathcal{L}}_G\right)^m\phi(x)|dx~\nonumber\\~
&\leq&  C A^{2m}\left((2m)!\right)^s\end{eqnarray*}
for all $[\xi]\in \Gh$ and $m\in\N_{0}.$
Recalling the definition of $\ell^{\infty}(\Gh)$ in \eqref{EQ:linf} we obtain
\eqref{EQ:FT1}.
\end{proof} 

We can now prove Theorem \ref{THM:Gevrey}.

\begin{proof}[Proof of Theorem \ref{THM:Gevrey}]
\textbf{(R)}  ``Only if'' part. \\
Let $\phi\in \gamma_{s}(G).$ 
Using $k!\leq k^k$ and Lemma \ref{L:FT} we get
\begin{equation}\label{EQ:ft1}
||\widehat{\phi}(\xi)||_{\HS}\leq C d_{\xi}^{1/2} \inf_{2m\geq 0}|\xi|^{-2m} A^{2m}\left(2m\right)^{2ms}\end{equation}
for all $[\xi]\in\Ghs.$
We will show that this implies the (sub-)exponential decay in \eqref{EQ:GR}.
It is known that for $r>0,$ we have the identity
\begin{equation}\label{EQ:x0}
\inf_{x> 0}x^{sx}r^{-x}=e^{-(s/e)r^{1/s}}.
\end{equation}
So for a given $r>0$ there exists some $x_0=x_{0}(r)>0$ such that  
\begin{equation}\label{EQ:x00}
\inf_{x>0}x^{sx}\left(\frac{r}{8^s}\right)^{-x}=x_{0}^{sx_0}\left(\frac{r}{8^s}\right)^{-x_0}.
\end{equation}
We will be interested in large $r$, in fact we will later set 
$r=\frac{|\xi|}{A}$, so we can assume that $r$ is large. Consequently, in
\eqref{EQ:x0} and later, we can assume that $x_{0}$ is sufficiently large.
Thus, we can take an even (sufficiently large) integer $m_{0}$ such that
$m_0\leq x_0<m_0+2$. Using the trivial inequalities
$$
\left(m_0\right)^{sm_0}r^{-(m_0+2)}\leq x_0^{sx_0}r^{-x_0},\quad r\geq 1,
$$ 
and 
$$
\left(k+2\right)^{k+2}\leq 8^k k^k
$$ 
for any $k\geq 2,$  we obtain 
$$
\left(m_0+2\right)^{s(m_0+2)}r^{-(m_0+2)}\leq 8^{sm_0}m_0^{sm_0}r^{-(m_0+2)}
\leq x_{0}^{sx_0}\left(\frac{r}{8^s}\right)^{-x_0}.
$$  
It follows from this, \eqref{EQ:x0} and \eqref{EQ:x00}, that 
\begin{equation}\label{EQ:ft3}
\inf_{2m\geq 0}{(2m)^{2sm}r^{-2m}}\leq x_0^{sx_0}\left(\frac{r}{8^s}\right)^{-x_0}
=e^{-(s/e)(\frac{r}{8^s})^{1/s}}.
\end{equation}
Let now $r=\frac{|\xi|}{A}$. From \eqref{EQ:ft1} and \eqref{EQ:ft3}
we obtain
\begin{eqnarray}
\|\widehat{\phi}(\xi)\|_{\HS}&\leq&C d_{\xi}^{1/2}
\inf_{2m\geq 0}\frac{A^{2m}}{|\xi|^{2m}}\left(2m\right)^{2ms}~\nonumber\\~
&=&C d_{\xi}^{1/2} \inf_{2m\geq 0}r^{-2m}\left(2m\right)^{2ms}~\nonumber\\~
&\leq&C d_{\xi}^{1/2} e^{-(s/e)\left(\frac{r}{8^s}\right)^{1/s}}\nonumber\\~
&=&C d_{\xi}^{1/2} e^{-(s/e)\frac{|\xi|^{1/s}}{8A^{1/s}}}~\nonumber\\~
&\leq&C d_{\xi}^{1/2} e^{-2B|\xi|^{1/s}},
\label{aux5}
\end{eqnarray}
with $2B=\frac{s}{8e}\frac{1}{A^{1/s}}.$ From \eqref{EQ:exp} it follows that
$d_{\xi}^{1/2} e^{-B|\xi|^{1/s}}\leq C.$ 
Using \eqref{EQ:ll}, we obtain \eqref{EQ:GR} for all $[\xi]\in\Ghs.$
On the other hand, for trivial $[\xi]\in\Gh\backslash\Ghs$ the estimate
\eqref{EQ:GR} is just the condition of the boundedness. 
This completes the proof of the ``only if'' part.

Now we prove the ``if'' part. 
Suppose $\phi\in C^{\infty}(G)$ is such that \eqref{EQ:GR} holds, i.e. we have
$$||\widehat{\phi}(\xi)||_{\HS}\leq K e^{-B\jp{\xi}^{1/s}}.$$
The $\ell^{1}(\Gh)-L^{\infty}(G)$ boundedness of the inverse Fourier transform in \eqref{EQ:HY} 
implies
\begin{eqnarray}
\|(-\L_G)^{k}\phi\|_{L^{\infty}(G)}&\leq& \| |\xi|^{2k}\widehat\phi\|_{\ell^{1}(\Gh)}\nonumber\\~
&=&\sumxi d_{\xi}^{3/2} |\xi|^{2k} ||\widehat{\phi}(\xi)||_{\HS}\nonumber\\~
&\leq&K\sumxi d_{\xi}^{3/2}\jp{\xi}^{2k}e^{-B\jp{\xi}^{1/s}}\nonumber\\~
&\leq&K\sumxi d_{\xi}^{3/2} e^{\frac{-B\jp{\xi}^{1/s}}{2}} \p{\jp{\xi}^{2k} e^{\frac{-B\jp{\xi}^{1/s}}{2}}}.
\label{aux2}
\end{eqnarray} 
Now we will use the following simple inequality, $\frac{t^N}{N!}\leq e^{t}$ for $t>0$. 
Setting later $m=2k$ and $a=\frac{B}{2}$, we estimate
$$
(m!)^{-s}\jp{\xi}^{m}=\p{\frac{(a\jp{\xi}^{1/s})^{m}}{m!}}^{s} a^{-sm} \leq
a^{-sm} e^{a\jp{\xi}^{1/s}},
$$
which implies 
$e^{-\frac{B}{2}\jp{\xi}^{1/s}}\jp{\xi}^{2k}\leq A^{2k}((2k)!)^{s},$
with $A=a^{-s}=(2/B)^{s}.$ 
Using this inequality and \eqref{aux2} we obtain
\begin{equation}\label{aux3}
\|(-\L_G)^{k}\phi\|_{L^{\infty}}\leq K\sumxi d_{\xi}^{3/2}
e^{\frac{-B\jp{\xi}^{1/s}}{2}}A^{2k}((2k)!)^s
\leq C A^{2k}((2k)!)^s
\end{equation} 
with $A= \frac{2^s}{B^s}$, where the convergence of the series in $[\xi]$ follows from
Lemma \ref{L:series}.
Therefore, $\phi\in \gamma_{s}(G)$ by Proposition \ref{PROP:l}.
\medskip

\textbf{(B)}  ``Only if'' part.  
Suppose  $\phi\in \gamma_{(s)}(G).$ 
For any given $B>0$ define $A$ by solving $2B=\left(\frac{s}{8e}\right)\frac{1}{A^{1/s}}.$ 
By Lemma \ref{L:FT} there exists $K_{B}>0$ such that 
\begin{equation*}\label{EQ:ft2}
||\widehat{\phi}(\xi)||_{\HS}\leq K_{B} d_{\xi}^{1/2} \inf_{2m\geq 0} |\xi|^{-2m} A^{2m}\left(2m\right)^{2ms}.\end{equation*}
Consequently, arguing as in case \textbf{(R)} we get \eqref{aux5}, i.e. 
\begin{equation*}\label{aux4}
\|\widehat{\phi}(\xi)\|_{\HS}\leq 
K_B d_{\xi}^{1/2}e^{-2B |\xi|^{1/s}} 
\end{equation*} 
for all $[\xi]\in \Gh.$ The same argument as in the case \textbf{(R)} now completes the proof.

``If'' part. 
For a given $A>0$ define $B>0$ by solving  $A=\frac{2^s}{B^s}$ and take $C_A$ big enough as 
in the case of \textbf{(R)}, so that we get
$$
\|(-\L_G)^{k}\phi\|_{L^{\infty}}\leq C_A A^{2k}((2k)!)^s.
$$ 
Therefore, $\phi\in \gamma_{(s)}(G)$ by Proposition \ref{PROP:l}.
\end{proof}

\section{$\alpha$-duals $\gamma_s(G)^{\wedge}$ and $\gamma_{(s)}(G)^{\wedge}$, 
for any $s$, $0<s<\infty$.}
\label{SEC:alpha}

First we analyse $\alpha$-duals of Gevrey spaces regarded as sequence spaces through
their Fourier coefficients.

We can embed $\gamma_{s}(G)\left(\textrm{or}~\gamma_{(s)}(G)\right)$ 
in the sequence space $\Sigma$ using the Fourier coefficients and
Theorem \ref{THM:Gevrey}. 
We denote the $\alpha$-dual of such the sequence space 
$\gamma_{s}(G)$ (or $\gamma_{(s)}(G)$) as 
\begin{multline*}
[\gamma_{s}(G)]^{\wedge}
=\left\{v=(v_{\xi})_{[\xi]\in\Gh}\in\Sigma:
\sumxi\sumij  |(v_{\xi})_{ij}| |\widehat{\phi}(\xi)_{ij}|<\infty
\textrm{ for all } \phi\in\gamma_{s}(G)
\right\},
\end{multline*}
with a similar definition for $\gamma_{(s)}(G).$

\begin{lemma} \label{L:ser}
{\rm\bf (R)} We have 
$v \in \left[ {\gamma_s \left( G \right)} \right]^{\wedge}$ 
if and only if for every $B>0$ the inequality
\begin{equation}\label{EQ:ad}
\sumxi e^{ - B\jp{\xi}^{\frac{1} {s}} } \left\| {v_\xi}
\right\|_{\HS}   < \infty
\end{equation}
holds for all $[\xi]\in\Gh.$ \\
{\rm\bf (B)} Also, we have 
$v \in \left[ {\gamma_{(s)} \left( G \right)} \right]^{\wedge}$ 
if and only if there exists $B>0$ such that the inequality
\eqref{EQ:ad}
holds for all $[\xi]\in\Gh.$
\end{lemma}
The proof of this lemma in (R) and (B) cases will be different. 
For (R) we can show this
directly, and for (B) we employ the theory of echelon spaces by K\"othe \cite{Koe}.

\begin{proof}
\textbf{(R)} ``Only if'' part. 
Let $v \in \left[ {\gamma_s \left( G \right)} \right]^{\wedge}$.
For any $B>0$, define $\phi$ by setting its Fourier coefficients to be
$\widehat{\phi}(\xi)_{ij}:=d_\xi e^{ -B\jp{\xi}^{\frac{1} {s}} },$
so that
$\|\widehat{\phi}(\xi)\|_{\HS}=d_\xi^2 e^{ -B\jp{\xi}^{\frac{1} {s}} }\leq
Ce^{ -\frac{B}{2}\jp{\xi}^{\frac{1} {s}} }$
by \eqref{EQ:exp}, which implies that
$\phi \in \gamma_s \left( G \right)$ by Theorem \ref{THM:Gevrey}.
Using Lemma \ref{L:c}, we obtain 
\[
\sumxi  e^{-B\jp{\xi}^{\frac{1}{s}}}  \|{v_\xi}\|_{\HS} \leq
\sumxi  d_{\xi} e^{-B\jp{\xi}^{\frac{1}{s}}}  \|{v_\xi}\|_{\ell^1(\C)} =
\sumxi \sumij  |(v_\xi)_{ij}| |\widehat{\phi}(\xi)_{ij}|  <
\infty
\]
by the assumption  $v \in \left[ {\gamma_s \left( G \right)} \right]^{\wedge}$,
proving the ``only if'' part. \\

``If'' part.
Let $ \phi \in \gamma_s(G)$. Then by Theorem \ref{THM:Gevrey}
there exist some  $B>0$ and $C>0$ such that 
$$
\|\widehat{\phi}(\xi)\|_{\HS} \leq C
e^{-B\jp{\xi}^{\frac{1}{s}}},
$$ 
which implies that
$$ 
\sumxi \sumij |(v_{\xi})_{ij}| |\widehat{\phi}(\xi)_{ij}| 
\leq \sumxi \|v_\xi\|_{\HS} \|\widehat{\phi}(\xi)\|_{\HS}
\leq C \sumxi e^{-B\jp{\xi}^{\frac{1}{s}}} \|v_\xi\|_{\HS} 
  < \infty
$$
is finite by the assumption \eqref{EQ:ad}.
But this means that $v \in [\gamma_s(G)]^{\wedge}.$ \\

\textbf{(B)}  
For any $B>0$ we consider the so-called echelon space,
$$
E_B=\left\{v=(v_\xi)\in\Sigma: \sumxi\sumij 
e^{- B\jp{\xi}^{\frac{1} {s}} } 
|({v_\xi})_{ij}|<\infty\right\}.
$$ 
Now, by diagonal transform we have $E_B \cong l^1$  and hence 
$\widehat{E_{B}}\cong l^{\infty}$, and it is easy to check 
that $\widehat{E_{B}}$ is given by 
$$
\widehat{E_{B}}=
\left\{w=(w_{\xi})\in\Sigma\;|\; \exists K>0\; : \quad
|(w_\xi)_{ij}|\leq K e^{-B\jp{\xi}^{1/s}}
\textrm{ for all } 1\leq i,j\leq d_\xi
\right\}.
$$ 
By Theorem \ref{THM:Gevrey} we know that
$\phi\in \gamma_{(s)}(G)$ if and only if 
$\left(\widehat{\phi}(\xi)\right)_{[\xi]\in \Gh}
\in \bigcap_{B>0}{\widehat{E_B}}.$ 
Using K\"othe's theory relating echelon and co-echelon spaces 
\cite[Ch. 30.8]{Koe}, we have, consequently, that
$v\in \gamma_{(s)}(G)^{\wedge}$ if and only if 
$(v_{\xi})_{[\xi]\in \Gh}\in \bigcup_{B>0}E_{B}$.
But this means that for some $B>0$ we have
$$
\sumxi\sumij 
e^{- B\jp{\xi}^{\frac{1} {s}} } 
|({v_\xi})_{ij}| < \infty.
$$
Finally, we observe that this is equivalent to
\eqref{EQ:ad} if we use
Lemma \ref{L:c} and \eqref{EQ:exp}.
\end{proof}

We now give the characterisation for $\alpha$-duals.

\begin{thm}\label{THM:aduals}
Let $0<s<\infty$. \\
{\rm\bf (R)} We have $v \in \left[ {\gamma_s \left( G
\right)} \right]^\wedge$ if and only if for every $B>0$ there exists $K_B>0$ such that
\begin{equation}\label{EQ:ade}
|| { v_ \xi  }||_{\HS}  \leq K_B e^{B\left\langle \xi  \right\rangle ^{\frac{1} {s}} }
\end{equation}
holds for all $ [\xi] \in \Gh$. \\
{\rm\bf (B)} We have $v \in \left[ {\gamma_{(s)} \left( G
\right)} \right]^\wedge$ 
if and only if there exist $B>0$ and $K_{B}>0$ such that
\eqref{EQ:ade}
holds for all $ [\xi] \in \Gh$.
\end{thm}

\begin{proof}
We prove the case (R) only since the proof of (B) is similar.
First we deal with ``If'' part. Let $v\in\Sigma$ be such that 
\eqref{EQ:ade} holds for every $B>0$. Let $\va\in\gamma_{s}(G)$. Then by
Theorem \ref{THM:Gevrey} there exist 
some constants $A>0$ and $C>0$ such that
$\|\widehat{\phi}(\xi)\|_{\HS}\leq C e^{-A\jp{\xi}^{1/s}}.$ 
Taking $B=A/2$ in 
\eqref{EQ:ade} we get that
\begin{eqnarray*}
\sumxi \sumij |(v_{\xi})_{ij}| |\widehat{\phi}(\xi)_{ij}| 
\leq
\sumxi \|v_{\xi}\|_{\HS} \|\widehat{\phi}(\xi)\|_{\HS}
\leq C K_B \sumxi e^{-\frac{A}{2}\jp{\xi}^{1/s}}<\infty,
\end{eqnarray*} 
so that $v \in \left[ {\gamma_s \left( G
\right)} \right]^\wedge$. \\
``Only if'' part. 
Let $v\in \br{\gamma_{s}(G)}^{\wedge}$ and let $B>0$.  
Then by Lemma \ref{L:ser} 
we have that 
$$
\sumxi e^{-B\jp{\xi}^{1/s}}||v_\xi||_{\HS}<\infty.$$ 
This implies that the exists a constant $K_{B}>0$ such that 
$e^{-B\jp{\xi}^{1/s}}||{v}(\xi)||_{\HS}\leq K_{B},$
yielding \eqref{EQ:ade}.
\end{proof}

We now want to show that the Gevrey spaces are perfect in the sense of K\"othe.
We define the $\alpha-$dual of $[\gamma_{s}(G)]^{\wedge}$ as 
$$
[\widehat{\gamma_{s}(G)}]^{\wedge}=
\left\{w=(w_{\xi})_{[\xi]\in\Gh}\in\Sigma:
\sumxi \sumij |(w_{\xi})_{ij}| |(v_\xi)_{ij}|<\infty 
\textrm{ for all } v\in[\gamma_{s}(G)]^{\wedge}\right\},
$$
and similarly for $[\gamma_{(s)}(G)]^{\wedge}$.
First, we prove the following lemma.

\begin{lemma}\label{L:perfect}
{\rm\bf (R)} We have $w \in \left[ {\widehat{\gamma_s \left( G
\right)}} \right]^\wedge$ if and only if there exists $B>0$ such that
\begin{equation}\label{EQ:pc}
\sumxi e^{B\jp{\xi}^{\frac{1} {s}}} \| {w_\xi}\|_{\HS}   < \infty.
\end{equation}
{\rm\bf (B)} We have $w \in \left[ {\widehat{\gamma_{(s)} \left( G
\right)}} \right]^\wedge$ if and only if for every $B>0$ the series
\eqref{EQ:pc} converges.
\end{lemma}
\begin{proof}  
We first show the Beurling case as it is more straightforward. \\
\textbf{(B)} ``Only if'' part. We assume that $w \in \left[ {\widehat{\gamma_{(s)} \left( G
\right)}} \right]^\wedge$. Let $B>0$, and define
$(v_{\xi})_{ij}:=d_{\xi} e^{B\jp{\xi}^{\frac{1} {s}}}.$ 
Then $\|v_{\xi}\|_{\HS}=d_{\xi}^{2} e^{B\jp{\xi}^{\frac{1} {s}}}\leq Ce^{2B\jp{\xi}^{\frac{1} {s}}}$ by
\eqref{EQ:exp}, which implies $v\in{[\gamma_{(s)}(G)]^{\wedge}}$ by
Theorem \ref{THM:aduals}.
Consequently, using Lemma \ref{L:c}
we can estimate
\[
\sumxi  e^{B\jp{\xi}^{\frac{1}{s}}}  \|{w_\xi}\|_{\HS} \leq
\sumxi  d_{\xi} e^{B\jp{\xi}^{\frac{1}{s}}} \sumij  |(w_\xi)_{ij}| =
\sumxi \sumij  |(v_\xi)_{ij}| |(w_{\xi})_{ij}|  <
\infty,
\]
implying \eqref{EQ:pc}. \\

``If'' part. Here we are given $w\in\Sigma$ such
that for every $B>0$ the series  \eqref{EQ:pc} converges.
Let us take any $v\in {[\gamma_{(s)}(G)]^{\wedge}}$. 
By Theorem \ref{THM:aduals} there exist $B>0$ and $K>0$ such that
$\|v_{\xi}\|_{\HS}\leq  K e^{B\jp{\xi}^{\frac{1} {s}}}.$ 
Consequently, we can estimate
\[
\sumxi \sumij  |(v_\xi)_{ij}| |(w_{\xi})_{ij}|  \leq
\sumxi  \|v_\xi\|_{\HS} \|w_{\xi}\|_{\HS} \leq 
K \sumxi  e^{B\jp{\xi}^{\frac{1}{s}}}  \|{w_\xi}\|_{\HS}<
\infty
\]
by the assumption \eqref{EQ:pc}, which shows that  
$w \in \left[ {\widehat{\gamma_{(s)} \left( G \right)}} \right]^\wedge$.

\textbf{(R)}
For $B>0$ we consider the echelon space
$$
D_{B}=
\left\{v=(v_{\xi})\in\Sigma\;|\; \exists K>0\; : \quad
|(v_\xi)_{ij}|\leq K e^{B\jp{\xi}^{1/s}}
\textrm{ for all } 1\leq i,j\leq d_\xi
\right\}.
$$ 
By diagonal transform we have $D_{B}\cong l^{\infty}$,
and since $l^{\infty}$ is a perfect sequence space, we have
$\widehat{D_{B}}\cong l^{1}$, and it is given by
$$
\widehat{D_{B}}=\left\{w=(w_\xi)\in\Sigma: \sumxi\sumij 
e^{B\jp{\xi}^{\frac{1} {s}} } 
|({w_\xi})_{ij}|<\infty\right\}.
$$ 
By Theorem \ref{THM:aduals}
we know that
$\gamma_{s}(G)^{\wedge}=\bigcap_{B>0} D_{B}$, and hence
$\left[ {\widehat{\gamma_{s} \left( G \right)}} \right]^\wedge
=\bigcup_{B>0} \widehat{D_{B}}.$
This means that 
$w\in \left[ {\widehat{\gamma_{s} \left( G \right)}} \right]^\wedge$
if and only if
there exists $B>0$ such that we have
$\sumxi \sumij e^{2B\jp{\xi}^{\frac{1}{s}}}   |({w_\xi})_{ij}|<\infty.$
Consequently, by Lemma \ref{L:c} we get
$$
\sumxi  e^{B\jp{\xi}^{\frac{1}{s}}}  \|{w_\xi}\|_{\HS}\leq 
\sumxi  d_{\xi} e^{B\jp{\xi}^{\frac{1}{s}}}  \|{w_\xi}\|_{\ell^{1}(\C)}\leq
C\sumxi \sumij e^{2B\jp{\xi}^{\frac{1}{s}}}   |({w_\xi})_{ij}|<\infty,
$$
completing the proof of the ``only if'' part. Conversely, given
\eqref{EQ:pc} for some $2B>0$, we have
$$
\sumxi \sumij e^{B\jp{\xi}^{\frac{1}{s}}}   |({w_\xi})_{ij}|\leq
\sumxi  d_{\xi }e^{B\jp{\xi}^{\frac{1}{s}}}  \|{w_\xi}\|_{\HS}\leq
C\sumxi  e^{2B\jp{\xi}^{\frac{1}{s}}}  \|{w_\xi}\|_{\HS}<\infty,
$$
implying $w\in\left[ {\widehat{\gamma_{s} \left( G \right)}} \right]^\wedge$.
\end{proof}

Now we can show that the Gevrey spaces are perfect spaces (sometimes called K\"othe spaces).

\begin{thm} 
$\gamma_{s}(G)$ and $\gamma_{(s)}(G)$ are perfect spaces, that is, 
$\gamma_{s}(G)=[\widehat{\gamma_{s}(G)}]^{\wedge}$ and
$\gamma_{(s)}(G)=[\widehat{\gamma_{(s)}(G)}]^{\wedge}.$
\end{thm}

\begin{proof} 
We will show this for ${\gamma_{s}(G)}$ since the proof for
${\gamma_{(s)}(G)}$ is analogous.
From the definition of $[\widehat{\gamma_{s}(G)}]^{\wedge}$ we have 
$\gamma_{s}(G)\subseteq [\widehat{\gamma_{s}(G)}]^{\wedge}.$ 
We will prove the other direction, i.e.,
$[\widehat{\gamma_{s}(G)}]^ {\wedge}\subseteq \gamma_{s}(G)$. 
Let $w={(w_{\xi})_{[\xi]\in{\Gh}}}\in [\widehat{\gamma_{s}(G)}]^ {\wedge}$ and
define 
$$
\phi(x):=\sumxi d_{\xi}\Tr \left(w_{\xi}\xi(x)\right).
$$ 
The series makes sense due to Lemma \ref{L:perfect}, and we have
$\|\widehat{\phi}(\xi)\|_{\HS}=\|w_{\xi}\|_{\HS}$. 
Now since $w\in [\widehat{\gamma_{s}(G)}]^{\wedge}$ 
by Lemma \ref{L:perfect} there exists $B>0$ such that
$\sumxi e^{  B\left\langle \xi
\right\rangle^{\frac{1} {s}} } || {w_\xi}||_{\HS}  < \infty$, 
which implies that for some $C>0$ we have
\begin{eqnarray*}
e^{B\jp{\xi}^{1/s}}||w_{\xi}||_{\HS}< C
&\Rightarrow& 
||\widehat{\phi}(\xi)||_{\HS}\leq C e^{-B\jp{\xi}^{1/s}}.
\end{eqnarray*}  
By Theorem \ref{THM:Gevrey} this implies $\phi\in \gamma_{s}(G).$
Hence $\gamma_{s}(G)=[\widehat{\gamma_{s}(G)}]^{\wedge},$ 
i.e.  $\gamma_{s}(G)$ is a perfect space. 
\end{proof}

\section{Ultradistributions $\gamma_s'(G)$ and $\gamma_{(s)}'(G)$}
\label{SEC:ultra}

Here we investigate the Fourier coefficients criteria for spaces of ultradistributions.
The space $\gamma_{s}'(G)$ (resp. $\gamma_{(s)}'(G)$) of the 
ultradistributions of order $s$ is defined as the dual of 
$\gamma_{s}(G)$ (resp. $\gamma_{(s)}(G)$) 
endowed with the standard inductive limit
topology of $\gamma_{s}(G)$ (resp. the projective limit topology of $\gamma_{(s)}(G)$).

 \begin{defi} \label{DEF:dual}
 The space $\gamma_{s}'(G)\left(\textrm{resp. } \gamma_{(s)}'(G)\right)$ 
 is the set of the linear forms $u$ on $\gamma_{s}(G)\left(\textrm{resp. } \gamma_{(s)}(G)\right)$ 
 such that for every $\epsilon>0$ there exists $C_\epsilon$ 
 (resp. for some $\epsilon>0$ and $C>0$) such that
  $$
  |u(\phi)|\leq C_{\epsilon}\sup_{\alpha}\epsilon^{|\alpha|}(\alpha!)^{-s}
  \sup_{x\in G}|(-\mathcal{L}_{G})^{|\alpha|/2}\phi(x)|
  $$ 
 holds for all $\phi\in \gamma_{s}(G)$ (resp. $\phi\in\gamma_{(s)}(G)$). 
 \end{defi}
 We can take the Laplace-Beltrami operator in Definition \ref{DEF:dual}
 because of the equivalence of norms given by Proposition \ref{PROP:l}.
 
We recall that for any $v\in \gamma_{s}'(G)$, for $[\xi]\in\Gh$, we define the Fourier coefficients 
$\widehat{v}(\xi):=\jp{v,\xi^{\ast}}\equiv v(\xi^{*}).$

We have the following theorem showing that topological and $\alpha$-duals of Gevrey
spaces coincide.

\begin{thm} \label{THM:equiv}
Let $1\leq s<\infty.$ Then
$v\in \gamma_{s}'(G)\left(\textrm{resp. } \gamma_{(s)}'(G)\right)$ if and only if 
$v\in \gamma_{s}(G)^{\wedge}\left(\textrm{resp. } \gamma_{(s)}(G)^{\wedge}\right).$ 
\end{thm}

\begin{proof} 
\textbf{(R)} ``If'' part. 
Let  $v\in \gamma_{s}(G)^{\wedge}.$ 
For any $\phi\in \gamma_{s}(G)$ define 
\begin{equation}\label{EQ:defv}
v(\phi):=\sumxi d_{\xi}\Tr\left(\widehat{\phi}(\xi)v_{\xi}\right).
\end{equation}
Since by Theorem \ref{THM:Gevrey} there exist some $B>0$ such that
$||\widehat{\phi}(\xi)||_{\HS}\leq C e^{-B\jp{\xi}^{1/s}}$,
we can estimate
$$
\sumxi d_{\xi}\Tr\left(\widehat{\phi}(\xi)v_{\xi}\right) \leq
\sumxi d_{\xi} \|\widehat{\phi}(\xi)\|_{\HS} \|v_{\xi}\|_{\HS} \leq
C\sumxi d_{\xi}  e^{-B\jp{\xi}^{1/s}} \|v_{\xi}\|_{\HS}<\infty
$$
by Lemma \ref{L:ser} and \eqref{EQ:exp}.
Therefore, $v(\phi)$ in \eqref{EQ:defv} is a well-defined linear functional
on $ \gamma_{s}(G)$.
It remains to check that $v$ is continuous.    
Suppose $\phi_j\rightarrow \phi$ in $\gamma_{s}(G)$ as $j\to\infty$, that is, 
in view of Proposition \ref{PROP:l},
there is a constant $A>0$ such that
$$
\sup_{\alpha}A^{-|\alpha|}(\alpha!)^{-s}\sup_{x\in G}
|(-\mathcal{L}_{G})^{|\alpha|/2}(\phi_j(x)-\phi(x))|\rightarrow 0
$$ 
as $j\rightarrow \infty.$ 
It follows that 
$$
\|\left(-{\mathcal{L}}_{G}\right)^{|\alpha|/2}(\phi_j-\phi)\|_{\infty}\leq 
C_j A^{|\alpha|}\left((|\alpha|)!\right)^s,
$$ 
for a sequence $C_j\rightarrow 0$ as $j\rightarrow \infty.$   
From the proof of Theorem \ref{THM:Gevrey} it follows that
we then have 
$$
\|\widehat{\phi_j}(\xi)-\widehat \phi(\xi)\|_{\HS}\leq 
K_{j} e^{-B\jp{\xi}^{1/s}},
$$ 
where $B>0$ and $K_j\rightarrow 0$ as $j\rightarrow \infty.$  
Hence we can estimate 
\begin{eqnarray*}
|v(\phi_j-\phi)|&\leq&
\sumxi d_{\xi} \|\widehat{\phi_j}(\xi)-\widehat\phi(\xi)\|_{\HS}
\|{v}_(\xi)\|_{\HS}\nonumber\\
&\leq& K_{j} \sumxi d_{\xi} e^{-B\jp{\xi}^{1/s}} \|v_{\xi}\|_{\HS}
\rightarrow 0\nonumber 
\end{eqnarray*}  
as $j\rightarrow \infty$ since $K_j\rightarrow 0$ as $j\rightarrow \infty$ and 
$\sumxi d_{\xi} e^{-B\jp{\xi}^{1/s}} \|v_{\xi}\|_{\HS}<\infty$ by
Lemma \ref{L:ser} and \eqref{EQ:exp}.  
Therefore, we have $v\in \gamma_{s}'(G).$ \\

``Only if'' part. Let us now take $v\in \gamma_{s}'(G).$ 
This means that for every $\epsilon>0$ there exists $C_\epsilon$ such that
 $$
 |v(\phi)|\leq C_{\epsilon}\sup_{\alpha}\epsilon^{|\alpha|}(\alpha!)^{-s}\sup_{x\in G}
 |(-\mathcal{L}_{G})^{|\alpha|/2}\phi(x)|
 $$ 
 holds for all $\phi\in \gamma_{s}(G).$ 
 So then, in particular, we have
\begin{eqnarray*} 
|v(\xi^{\ast}_{ij})| &\leq& C_{\epsilon}\sup_{\alpha}\epsilon^{|\alpha|}(\alpha!)^{-s}
\sup_{x\in G}|(-\L_{G})^{|\alpha|/2}\xi^{\ast}_{ij}(x)|\nonumber\\
  &=& C_\epsilon \sup_{\alpha}\epsilon^{|\alpha|}(\alpha!)^{-s}|\xi|^{|\alpha|}\sup_{x\in G}|\xi^{\ast}_{ij}(x)| \\
  & \leq & C_\epsilon \sup_{\alpha}\epsilon^{|\alpha|}(\alpha!)^{-s}\langle\xi\rangle^{|\alpha|}\sup_{x\in G}
  \|\xi^{\ast}(x)\|_{\HS} \\
  & = & C_\epsilon \sup_{\alpha}\epsilon^{|\alpha|}(\alpha!)^{-s}\langle\xi\rangle^{|\alpha|} 
  d_{\xi}^{1/2}.
\end{eqnarray*}
This implies
\begin{eqnarray*}
||v({\xi^{\ast}})||_{\HS} =
\sqrt{{\sum_{i,j=1}^{d_{\xi}}|v(\xi^{\ast}_{ij})|^2}}
\leq  
C_\epsilon d_{\xi}^{3/2} \sup_{\alpha}\epsilon^{|\alpha|}(\alpha!)^{-s}\langle\xi\rangle^{|\alpha|}.
\end{eqnarray*}
Setting $r=\epsilon\langle\xi\rangle$ and using inequalities
$$
\alpha!\geq|\alpha|!n^{-|\alpha|}
\textrm{ and }
\p{\frac{(r^{1/s}n)^{|\alpha|}}{|\alpha|!}}^{s}\leq
\p{e^{r^{1/s}n}}^{s}=e^{n s r^{1/s}},
$$
we obtain 
\begin{eqnarray}
\|v(\xi^{\ast})\|_{\HS} &\leq&
C_{\epsilon} d_{\xi}^{3/2} \sup_{\alpha}\left(rn^{s}\right)^{|\alpha|}\left(|\alpha|!\right)^{-s}\nonumber\\
&\leq & C_{\epsilon}d_{\xi}^{3/2} \sup_{\alpha}e^{ns r^{1/s}}\nonumber\\
&=&C_\epsilon d_{\xi}^{3/2} e^{ns \epsilon^{1/s}\jp{\xi}^{1/s}} 
\end{eqnarray} 
for all $\epsilon>0.$
We now recall that $v(\xi^{*})=\widehat{v}(\xi)$ and, therefore,
with $v_{\xi}:=\widehat{v}(\xi)$, we get
$v\in \gamma_s(G)^{\wedge}$ by Theorem \ref{THM:aduals}
and \eqref{EQ:exp}. \\

\textbf{(B)} This case is similar but we give the proof for completeness. \\
 ``If'' part. 
Let  $v\in \gamma_{(s)}(G)^{\wedge}$ and  
for any $\phi\in \gamma_{(s)}(G)$ define 
$v(\phi)$ by \eqref{EQ:defv}. By a similar argument to the case (R),
it is a well-defined linear functional on $\gamma_{(s)}(G)$.
To check the continuity,
suppose $\phi_j\rightarrow \phi$ in $\gamma_{(s)}(G)$, that is, 
for every $A>0$ we have
$$
\sup_{\alpha}A^{-|\alpha|}(\alpha!)^{-s}\sup_{x\in G}
|(-\mathcal{L}_{G})^{|\alpha|/2}(\phi_j(x)-\phi(x))|\rightarrow 0
$$ 
as $j\rightarrow \infty.$ 
It follows that 
$$
\|\left(-{\mathcal{L}}_{G}\right)^{|\alpha|/2}(\phi_j-\phi)\|_{\infty}\leq 
C_j A^{|\alpha|}\left((|\alpha|)!\right)^s,
$$ 
for a sequence $C_j\rightarrow 0$ as $j\rightarrow \infty,$ for every $A>0.$   
From the proof of Theorem \ref{THM:Gevrey} it follows that
for every $B>0$ we have 
$$
\|\widehat{\phi_j}(\xi)-\widehat \phi(\xi)\|_{\HS}\leq 
K_{j} e^{-B\jp{\xi}^{1/s}},
$$ 
where $K_j\rightarrow 0$ as $j\rightarrow \infty.$  
Hence we can estimate 
\begin{eqnarray*}
|v(\phi_j-\phi)|&\leq&
\sumxi d_{\xi} \|\widehat{\phi_j}(\xi)-\widehat\phi(\xi)\|_{\HS}
\|{v}_{\xi}\|_{\HS}\nonumber\\
&\leq& K_{j} \sumxi d_{\xi} e^{-B\jp{\xi}^{1/s}} \|v_{\xi}\|_{\HS}
\rightarrow 0\nonumber 
\end{eqnarray*}  
as $j\rightarrow \infty$ since $K_j\rightarrow 0$ as $j\rightarrow \infty$, and 
where we now take $B>0$ to be such that
$\sumxi d_{\xi} e^{-B\jp{\xi}^{1/s}} \|v_{\xi}\|_{\HS}<\infty$ by
Lemma \ref{L:ser} and \eqref{EQ:exp}.  
Therefore, we have $v\in \gamma_{(s)}'(G).$ \\

``Only if'' part. Let $v\in \gamma_{(s)}'(G).$ This means that 
there exists $\epsilon>0$ and $C>0$ such that
$$ 
|v(\phi)| \leq C\sup_{\alpha}\epsilon^{|\alpha|}(\alpha!)^{-s}
\sup_{x\in G}|(-\mathcal{L}_{G})^{|\alpha|/2}\phi(x)|
$$ 
holds for all $\phi\in \gamma_{(s)}(G).$ 
Then, proceeding as in the case (R), we obtain
\begin{equation}
\|v(\xi^{\ast})\|_{\HS}\leq Cd_{\xi}^{3/2} e^{ns\epsilon^{1/s}\jp{\xi}^{1/s}},
\end{equation}
i.e. 
$\|\widehat{v}(\xi)\|_{\HS}\leq C e^{\delta\jp{\xi}^{1/s}},$ for some
$\delta>0$.
Hence $v\in \gamma_{(s)}(G)^{\wedge}$ by  by Theorem \ref{THM:aduals}.
\end{proof}

\end{document}